\documentclass[a4paper, reqno]{amsart}
\usepackage{amsmath}
\usepackage{amscd}
\usepackage{amsfonts}
\usepackage{enumerate}
\usepackage{comment}

\usepackage[all]{xy}

\newtheorem{thm}{Theorem}[section]
\newtheorem{lem}[thm]{Lemma}

\newtheorem{cor}[thm]{Corollary}

\newtheorem{defn}[thm]{Definition}
\newtheorem{ntn}[thm]{Notation}

\numberwithin{equation}{section}

\newcommand{\BC}{\mathbb{C}}
\newcommand{\BN}{\mathbb{N}}
\newcommand{\BR}{\mathbb{R}}

\newcommand{\BM}{\mathbb{M}}

\newcommand{\CL}{\mathcal{L}}

\newcommand{\KI}{\mathfrak{I}}
\newcommand{\KH}{\mathfrak{H}}

\newcommand{\CE}{\mathcal{E}}

\newcommand{\CB}{\mathrm{CB}}

\newcommand{\os}{\mathrm{os}}
\newcommand{\rd}{\mathrm{d}}

\newcommand{\Morc}{\mathrm{CCP}}
\newcommand{\sa}{\mathrm{sa}}

\newcommand{\M}{\mathrm{A}}
\newcommand{\N}{\mathrm{N}}

\newcommand{\smnoind}{\noindent}

\begin{document}
	\title{Predual and tight operator systems}
	
	\author{Chi-Keung Ng and Xin-Rui Tan}
	
	\address[Chi-Keung Ng]{Chern Institute of Mathematics and LPMC, Nankai University, Tianjin 300071, China.}
	\email{ckng@nankai.edu.cn}

	\address[Xin-Rui Tan]{Chern Institute of Mathematics, Nankai University, Tianjin 300071, China.}
	\email{1120230017@mail.nankai.edu.cn}

	\date{\today}
	
	
	\begin{abstract}
		Given a (not necessarily unital) complete operator system $S$ with a generating cone, there are some  studies in literature (see, e.g., \cite{JN2, Ng}) on the operator system dual $S^\mathrm{d}$ of $S$, i.e., the dual matrix-ordered space $S^*$ equipped with a matrix norm that turns it into a dual operator system satisfying certain universal property.
		
In order to do further study on $S^\mathrm{d}$, one needs to consider the operator system predual construction. 
		More precisely, given a dual operator system $V$ with a generating cone, there is a matrix norm on the predual space $V_*$, that turns it into a complete operator system $V_\#$ satisfying certain universal property.
		
		In this article, we show that $V\cong T^\mathrm{d}$ for a complete operator system $T$ if and only if $V$ satisfies a natural property called tightness; in this case, $V\cong (V_\#)^\rd$. 
		Furthermore, we establish that if $S$ is tight, then $S\cong W_\#$ for a dual operator system $W$; in fact $S\cong (S^\mathrm{d})_\#$. 

		Since all unital complete operator systems and all $C^*$-algebra are tight, one sees that every unital complete operator system and every $C^*$-algebra is of the form  $W_\#$, for a dual operator system $W$. 
		%Moreover, if $\CE$ is the operator system associated with a tolerance relation and $\tilde \CE$ is its completion  (which is again an operator system), then $\tilde \CE\cong (\tilde \CE^\rd)_\#$. 
	\end{abstract}
	
	\maketitle
	
	\section{Introduction}
	
	%\medskip
	
	The study of duality for operator systems began with the work \cite{CE} of Choi and Effros, in which it was shown that for every finite dimensional unital operator system $S$, one can find an Archimedean matrix order unit in the dual matrix ordered space $S^*$, turning it into a unital operator system. 
%	Some interests on the duality of finite dimensional unital operator systems come from its connection with quantum information theory and quantum graph theory (see, e.g., \cite{DSW} and \cite{GMS}). 

However, when $S$ is an infinite dimensional unital operator system, $S^*$ may not have an order unit any more (see, e.g., \cite[Example 3.3(c)]{Ng1}). 
	Therefore, one needs to go beyond the unital case. 
	
	%\medskip
	
	Werner began the study of not necessarily unital operator systems in \cite{Wern-subsp,Werner} and gave a characterization of them in \cite[Theorem 4.15]{Wern-subsp}. 
	This theory was applied by Han in \cite{Han} to study duality of infinite dimensional operator systems. 
	Moreover, based on \cite{Wern-subsp,Werner},  the first named author introduced and studied in \cite{Ng1} a duality framework for general complete operator systems. 
	In \cite{JN2}, this framework was further studied and was applied to the class of operator systems introduced by Connes and van Suijlekom in \cite{CvD,CvD2} (see also \cite{GS23}). 
	On the other hand, an equivalent form of not necessarily unital operator systems, in terms of noncommutative convex sets, was given recently by Kennedy, Kim and Manor in \cite{KKM}. 
	This work was then applied in \cite{HKM} to study ``duailzability'' of operator systems in the sense of \cite{Ng1} (see also \cite{Han}).

	%\medskip
	More precisely, for a complete operator system $S$ (which is not necessarily unital), a certain ``universal dual operator system envelope'', $S^\rd$, of $S^*$, called the \emph{operator system dual},  was introduced in \cite{Ng1}.
	We say that $S$ 
	\begin{enumerate}
		\item has a \emph{generating cone} if the linear span of the cone $S_+$ equals $S$;
		
		\item is \emph{dualizable} if the canonical map from $S^*$ to $S^\rd$ is an operator space isomorphism.
	\end{enumerate}
	It was shown in \cite{JN2,Ng1} that the construction $S\mapsto S^\rd$ induces a functor from the category of operator systems with generating cones (respectively, dualizable operator systems) to the category of dual operator systems with generating cones(respectively, dualizable dual operator systems), called the \emph{dual functor}. 

In order to study the duality of operator systems via the dual functor, it is necessary to consider its predual counterpart.
The aim of this article is to define and study ``operator system preduals''.

Using the results in the article, it is recently obtained in \cite{JNT} that the dual functor is an equivalence from the category of tight operator systems (see Definition \ref{defn:d-dual}; this includes all unital operator systems and all $C^*$-algebras) to the category of tight dual operator systems.

	%\medskip
	
	The following  are two interesting questions concerning this duality framework:
	\begin{enumerate}[Q1).]
		\item Which dual operator system can be written in the form $S^\rd$? 
		
		\item What is the relation between the dual operator system construction and the predual operator system construction?%\footnote{This question was asked by van Suijlekom during a conversation with the first named author.}  
	\end{enumerate}

	%\medskip
	
	The aim of this article is to give a study of these two questions. 
	We find that these two questions are both related to the notion of ``tightness'' for operator systems as will be defined in  Definition \ref{defn:d-dual}. 
	More precisely, for a dual operator system $V$ and a complete operator system $S$, both with generating cones, 
	\begin{enumerate}[(i).]
		\item if $V_\#$ is the ``universal complete operator system envelope'' of the predual $V_*$ of $V$, then $V_\#$ is always 
		tight (see Theorem \ref{thm:os-predual-tight}(a));
		
		\item $V$ is tight if and only if $V \cong (V_\#)^\rd$, which is equivalent to $V\cong T^\rd$ for a complete operator system $T$ (see Theorem \ref{thm:predual-dual}(b));
		
		\item $S$ is tight if and only if $S\cong (S^\rd)_\#$ (see Theorem \ref{thm:os-predual-tight}(b)).
	\end{enumerate}
	
	%\medskip

Since unital complete operator systems are tight (see \cite[Theorem 3.17(b)]{Ng1}), the last two statements above are applied in Corollary \ref{cor:unital} to the case when $V$ and $S$ are unital. 	
In particular, every unital dual operator system can be written as $T^\rd$ for a complete operator system $T$ with a generating cone (see Corollary \ref{cor:unital}(b)). 
Moreover, Statement (iii) above also applies to the completion of the operator system associated with a tolerance relation (see Corollary \ref{cor:compl-op-sys} and \ref{cor:toler}).

%\medskip

The above indicates certain duality between the construction $V\mapsto V_\#$ and the construction $S\mapsto S^\rd$. 
	The precise functorial connections between them will be studied in a further work (see \cite{JNT}).

	%\medskip

	\section{Preliminaries}\label{sec:prelim}
	
	%\medskip
	
	We first recall some basic notations and facts that are needed in our consideration. 
	In this article, all vector spaces are over the complex field, unless otherwise stated. 
	
	%\medskip
	
	Let $E$ and $F$ be  normed spaces. 
	Denote by $E^*$ the dual (Banach) space of $E$ and by $\kappa_E: E\to E^{**}$ the canonical embedding; i.e. 
	\begin{equation}\label{eqt:kappa}
		\kappa_E(x)(f):= f(x) \qquad (x\in E; f\in E^*).
	\end{equation}
	It is clear that for a bounded linear map $\phi:E\to F$, one has 
	\begin{equation}\label{eqt:kappa}
		\phi^{**}\circ \kappa_E = \kappa_F\circ \phi.
	\end{equation}
	If $E$ is a partially ordered $^*$-vector space, in the sense of \cite{CE}, then we say that $E$ \emph{has a generating cone} if $E_\sa = E_+ - E_+$, where
	$$E_\sa:= \{x\in E:x^* = x\}.$$
	A partially ordered $^*$-vector space $E$ is called an \emph{ordered normed space}, if it is equipped with a norm such that the involution $^*$ is isometric and the cone $E_+$ is norm-closed. 
	In this case, we set 
	$$B_E:= \{x\in E: \|x\| \leq 1 \} \quad \text{and}\quad B^+_E:= \{x\in E_+: \|x\| \leq 1 \}.$$

Let us first recall the following well-known fact. 	
	%\medskip
	
	\begin{lem}\label{lem:image-weak-st-cont-bdd-below}
		Let $E$ and $F$ be (both real or complex) Banach spaces.
		If $\Phi:E^*\to F^*$ is a weak$^*$-continuous linear map which  is bounded below, in the sense that  $\inf\big\{\|\Phi(f)\|: f\in E^*; \|f\| =1 \big\} > 0$, then $\Phi(E^*)$ is a weak$^*$-closed subspace of $F^*$ and $\Phi$ is a weak$^*$-homeomorphism from $E^*$ onto $\Phi(E^*)$. 
	\end{lem}
	
	%\medskip
	
	Denote $\BM_\infty(E) := \bigcup_{n\in \BN} \BM_n(E)$, where $\BM_n(E)$ is the vector space of $n\times n$ matrices with coefficients in $E$ such that elements in $\BM_n(E)$ are regarded as elements in  $\BM_{n+1}(E)$ with their last columns and last rows both being zero. 
	For simplicity, we set $\BM_n:= \BM_n(\mathbb{C})$ and $\BM_\infty:= \BM_\infty(\mathbb{C})$. 
	Let us define a pairing between $\BM_\infty(E^*)$ and $\BM_\infty(E)$ as follows:
		\begin{equation}\label{eqt:def-inf-pair}
		(f,x) :={\sum}^\infty_{k,l =1} f_{k,l}(x_{k,l}), 
	\end{equation}
where $f=[f_{k,l}]_{k,l}\in \BM_\infty(E^*)$ and $x=[x_{k,l}]_{k,l}\in \BM_\infty(E)$.
	This pairing induces a topology $\sigma\big(\BM_\infty(E^*), \BM_\infty(E)\big)$ on $\BM_\infty(E^*)$ in the usual way. 
	As in \cite{OPS} (or \cite{PaulBook}), each linear map $\varphi:E\to F$ will induce a $\BM_\infty$-module map $\varphi^{(\infty)}:\BM_\infty(E) \to \BM_\infty(F)$. 
	
	%\medskip
	
	Recall that an \emph{operator space} is a vector space $X$ equipped with a norm on $\BM_\infty(X)$ that satisfies certain conditions (see \cite{OPS}).  
	If $Y$ is another operator space, then a linear map $\varphi:X\to Y$ is said to be 
	\begin{enumerate}
\item \emph{completely bounded} if $\varphi^{(\infty)}$ is bounded; 

		\item \emph{completely contractive} (or a \emph{complete contraction}) if $\varphi^{(\infty)}$ is contractive; 
		\item \emph{completely isometric} (or a \emph{complete isometry}) if $\varphi^{(\infty)}$ is isometric. 
	\end{enumerate}
	If $Y$ is another operator space, we denote by $\CB(X,Y)$ the set of all completely bounded maps from $X$ to $Y$.

	%\medskip
	
	Part (a) of the following is a natural notion to study dual spaces of operator systems, which  is a generalization of the notion of ``matrix ordered operator spaces'' as introduced in \cite{Wern-subsp}.

	%\medskip

	\begin{defn}
		(a) An operator space $X$ is called a \emph{semi-matrix ordered operator space (SMOS)} if
		\begin{itemize}
			\item $X$ has a $^*$-vector space structure (see \cite{CE}) such that the induced involution on $\BM_\infty(X)$ is isometric; 
			\item  there is a norm-closed subset
			$\BM_\infty(X)_+\subseteq \BM_\infty(X)_{sa}$ (called the \emph{matrix cone} of $X$)
			satisfying
			\[\alpha^* x\alpha + \beta^* y \beta\in \BM_\infty(X)_+ \qquad (\alpha,\beta\in \BM_\infty; x,y\in \BM_\infty(X)_+).\]
		\end{itemize}
		
		\smnoind
		(b) When $X$ and $Y$ are SMOS, a linear map $\varphi:X\to Y$ with $\varphi(X_\sa)\subseteq Y_\sa$ is called 
		\begin{enumerate}
			\item \emph{completely positive} if $\varphi^{(\infty)}\big(\BM_\infty(X)_+\big) \subseteq \BM_\infty(Y)_+$;
			
			\item a \emph{complete order monomorphism} if $\varphi$ is injective and $\varphi^{(\infty)}\big(\BM_\infty(X)_+\big) = \varphi^{(\infty)}\big(\BM_\infty(X)\big)\cap \BM_\infty(Y)_+$; 
			
			\item a \emph{complete order isomorphism} if $\varphi$ is a surjective complete order monomorphism. 
		\end{enumerate}
		The set of all completely positive complete contractions from $X$ to $Y$ is denoted by $\Morc(X,Y)$. 
		
		\smnoind
		(c) A SMOS $S$ is called an \emph{operator system} (respectively, a \emph{unital operator system}) if there exist a Hilbert space $\KH$ and a completely isometric complete order monomorphism $\Lambda:S\to \CL(\KH)$ (respectively, such that $1\in \Lambda(S)$). 
	\end{defn}

	%\medskip
	
		For basic informations on unital operator systems, the reader may consult, e.g., \cite{OPS} or \cite{PaulBook}.

	%\medskip
	Let $X$ be a SMOS.
	It is easy to see that $\Morc(X,\BM_n)$ is a compact Hausdorff space under the point-norm topology. 
	We consider the $C^*$-algebra 
	$$ \M(X) := {\bigoplus}_{n\in \BN}^{\ell^\infty}\, C\big(\Morc(X,\BM_n);\BM_n\big),$$
	%Although $\M(X)$ is also a von Neumann algebra, we will never consider its weak$^*$-topology, and see it as a $C^*$-algebra only. 
	and denote 
	$$j_X:X\to  \M(X)$$ 
	to be the canonical map given by evaluations. 
	It is easy to see that $j_X$ is a completely positive complete contraction. 
	
	%\medskip

%\begin{rem}\label{rem:j-iota}
%	Note that $\Morc(X,\BM_n)$ is a compact Hausdorff space under the point-norm topology. 
%	We set 
%	$$\mathrm{A}(X):={\bigoplus}_{n\in \BN}C\big(\Morc(X,\BM_n);\BM_n\big).$$ 
%	Then $\mathrm{A}(X)$ is a unital $C^*$-subalgebra of $ \M(X)$ and the map $j_X:X\to  \M(X)$ in the above is the composition of the evaluation map $$j_X:X\to \mathrm{A}(X)$$ as in \cite[Relation (2.1)]{Ng} with the inclusion map from $\mathrm{A}(X)$ to $ \M(X)$. 
%	Therefore, we will not distinguish the two notions of $j_X$ in the above.
%\end{rem}

	%\medskip
	
	\begin{lem}\label{lem:norm-ind-by-iota}
		Let $X$ be a SMOS.
		
		\smnoind
		(a) $\big(j_X^{(\infty)}\big)^{-1}\big(\BM_\infty(\M(X))_+\big)\cap \BM_\infty(X)_\sa = \BM_\infty(X)_+$.

		\smnoind
		(b) $X$ is an operator system if and only if $j_X$ is completely isometric. 
	\end{lem}
	
	%\medskip
	
	In fact, part (a) above follows from \cite[Lemma 4.2(a)]{JN2}, % and Remark \ref{rem:j-iota}, 
	while part (b) follows from \cite[Lemma 4.2(c)]{JN2}. %and Remark \ref{rem:j-iota}. 

%\medskip
	
	For a complete SMOS $Y$, we identify 
	\begin{equation}\label{eqt:Mn(X*)}
		M_n(Y^*) = \CB(Y;\BM_n) \quad (n\in \BN)
	\end{equation}
	in the canonical way.
	If $Y^*$ is equipped with the weak$^*$-topology $\sigma(Y^*,Y)$, the canonical dual operator space structure, the involution given by $f^*(y) := \overline{f(y^*)}$ ($y\in Y$), and the following matrix cone 
	$$\BM_n(Y^*)_+:= \big\{ \varphi\in \CB(Y, \BM_n): \varphi \text{ is completely positive}\big\}\quad (n\in \BN),$$
	then we call $Y^*$ a \emph{dual SMOS}. 
	In this case,  $Y$ is called the \emph{predual} of $Y^*$.

%	If we set $V:=Y^*$, then we denote $V_*:=Y$. 
	
	%\medskip

	Note that the predual of a dual SMOS is unique  (as it is determined by the weak$^*$-topology). 
%	However, it is not the case if the weak$^*$-topology is not taken into account.
%	In fact, it is possible to have two different complete SMOS $X$ and $Y$ such that one can find a completely isometric complete order isomorphism from $X^*$ onto $Y^*$ (see \cite[Proposition 2.3]{BM-dual}).
%	In fact, in \cite[Proposition 2.3]{BM-dual}, both $X^*$ and $Y^*$ are even unital dual operator systems in the following sense.  

	%It is not hard to see that  
	%\begin{equation}\label{eqt:pos-cone}
	%\BM_1(Y^*_+) = Y^*_+ = \{f\in Y^*: f^* = f; f(Y_+)\subseteq \RP\}.
	%\end{equation}
	%%\medskip
	
	%\medskip
	
	\begin{defn}
		Let $Y$ be a complete SMOS. 
		We say that the dual SMOS $Y^*$ is a \emph{dual operator system} (respectively, \emph{unital dual operator system}) if there exist a Hilbert space $\KH$ and a weak$^*$-homeomorphic completely isometric complete order isomorphism $\Lambda$ from $Y^*$ onto a weak$^*$-closed subspace of $\CL(\KH)$ (respectively, such that  $1\in \Lambda(Y^*)$). 
	\end{defn}

	%\medskip
	Notice that in the definitions above, an operator system (respectively, a dual operator system) is a self-adjoint (respectively, and weak$^*$-closed) subspace of some $\CL(\KH)$, equipped with the induced matrix norm and the induced matrix cone (respectively, together with the induced weak$^*$-topology). 
	These definitions are slightly different from the ones in some literature, where the unital assumption is usually imposed (see e.g., \cite{BM-dual,OPS}).
	
	For complete SMOS $X$ and $Y$, the set of all weak$^*$-continuous completely positive complete contractions from $Y^*$ to $X^*$ is denoted by $\Morc_w(Y^*,X^*)$. 
From now on, we will use  the following canonical identifications (see \eqref{eqt:Mn(X*)}):
\begin{equation}\label{eqt:CCP-B+}
	\Morc(Y;\BM_n) = B^+_{\BM_n(Y^{*})} \quad \text{and} \quad \Morc_w(Y^*;\BM_n) = B^+_{\BM_n(Y)}.
\end{equation}

	Let us put
	$$ \N(Y^*) := {\bigoplus}_{n\in \BN}^{\ell^\infty}\, \ell^\infty\big(\Morc_w(Y^*,\BM_n);\BM_n\big) = {\bigoplus}_{n\in \BN}^{\ell^\infty}\, \ell^\infty\big(B^+_{\BM_n(Y)};\BM_n\big),$$
	and denote $\mu_{Y^*}:Y^*\to  \N(Y^*)$ to be the canonical map given by evaluations. 
	Notice that $\N(Y^*)$ is a von Neumann algebra with its predual being 
	$$ \N(Y^*)_* = {\bigoplus}_{n\in \BN}^{\ell^1}\, \ell^1\big(\Morc_w(Y^*,\BM_n);\BM_n^*\big) = {\bigoplus}_{n\in \BN}^{\ell^1}\, \ell^1\big(B^+_{\BM_n(Y)};\BM_n^*\big).$$
It is easy to see that  $\mu_{Y^*}$ is a weak$^*$-continuous completely positive complete contraction. 
	
	%\medskip

Since $\Morc_w(Y^*,\BM_n)\subseteq \Morc(Y^*,\BM_n)$, one obtains a canonical $^*$-homomorphism 
$Q_{Y^*}: \M(Y^*)\to  \N(Y^*)$, via truncation,  
satisfying 
\begin{equation}\label{eqt:mu=Q-j}
\mu_{Y^*} = Q_{Y^*}\circ j_{Y^*}.
\end{equation}

	\begin{lem}\label{lem:unit-ball-weak-st-dense}
	Suppose that $Y$ is a complete SMOS and $n\in \BN$. 
	
	\smnoind
	(a) $\kappa_Y^{(n)}\big(B_{\BM_n(Y)}\big)$ is $\sigma\big(\BM_n(Y^{**}), \BM_n(Y^*)\big)$-dense in $B_{\BM_n(Y^{**})}.$
	
	\smnoind
	(b) $\kappa_Y^{(n)}\big(B_{\BM_n(Y)}^+\big)$ is $\sigma\big(\BM_n(Y^{**}), \BM_n(Y^*)\big)$-dense in $B_{\BM_n(Y^{**})}^+.$

	\smnoind
	(c) $B_{\BM_n(Y^*)}$ is $\sigma\big(\BM_n(Y^*), \BM_n(Y)\big)$-compact.
\end{lem}
\begin{proof}
	(a) This is precisely \cite[Corollary 2.8]{JN2}. 
	
\smnoind
(b)	This is taken from \cite[Proposition 12]{JN}. 
	
	\smnoind
	(c) It follows from \cite[Lemma 4.1.1]{OPS} that one may identify $\BM_n(Y^*)$ with the dual space $T_n(Y)^*$ isometrically such that  the duality between $\BM_n(Y^*)$ and $T_n(Y)$ is given by 
	$$ 	[f,x] ={\sum}^n_{k,l =1} (f_{k,l},x_{k,l}) \quad (f=[f_{k,l}]_{k,l}\in \BM_n(Y^*); x=[x_{k,l}]_{k,l}\in T_n(Y)). $$
	Since $[\cdot, \cdot]$ coincides with the pairing $(\cdot, \cdot)$ as in \eqref{eqt:def-inf-pair} when $T_n(Y)$ is identified with $\BM_n(Y)$ as vector spaces, we see that $\sigma\big(\BM_n(Y^*), \BM_n(Y)\big)$ and $\sigma\big(\BM_n(Y^*), T_n(Y)\big)$ coincide.
	Thus, the conclusion follows from the Banach-Alaoglu theorem. 
\end{proof}

One can use part (b) above (see also \eqref{eqt:CCP-B+}) to obtain Relations (4.6) and (4.7) of \cite{JN2}; 
in particular, we have  
\begin{equation}\label{eqt:iota=j}
	\|\mu_{Y^*}^{(\infty)}(x)\| = \|j_{Y^*}^{(\infty)}(x)\|\qquad (x\in \BM_\infty(Y^*)),
\end{equation}
and
\begin{align}\label{eqt:iota-inv=j-inv}
	\BM_\infty(Y^*)_+ \ =\ &  \big(\mu_{Y^*}^{(\infty)}\big)^{-1} \big(\BM_\infty(\N(Y^*))_+\big) \cap \BM_\infty(Y^*)_\sa \nonumber \\
	\ =\ &  \big(j_{Y^*}^{(\infty)}\big)^{-1}\big(\BM_\infty(\M(Y^*))_+\big) \cap \BM_\infty(Y^*)_\sa,  
\end{align}
which produces 
\begin{equation*}\label{eqt:Q-j=iota}
	Q_{Y^*}\big(j_{Y^*}^{(\infty)}(\BM_\infty(Y^*))_+\big) =  \mu_{Y^*}^{(\infty)}(\BM_\infty(Y^*))_+.
\end{equation*}
From these, we have the following lemma. 

\begin{lem}\label{lem:Q}
The restriction $Q_{Y^*}|_{j_{Y^*}(Y^*)}$ is a completely isometric complete order isomorphism from $j_{Y^*}(Y^*)$ onto $\mu_{Y^*}(Y^*)$.
\end{lem}

%\medskip

The above also produces the following result (see \cite[Theorem 4.8(c)]{JN2}). 
In fact, the forward implication of this result is given by Lemma \ref{lem:norm-ind-by-iota}(b) and \eqref{eqt:iota=j}, while the backward implication is given by Lemma \ref{lem:image-weak-st-cont-bdd-below} (note that $\mu_{Y^*}$ is weak$^*$-continuous) and \eqref{eqt:iota-inv=j-inv}.

	%\medskip
	
	\begin{lem}\label{lem:mu-compl-ord-mono}
		Let $Y$ be a complete SMOS.
		Then $Y^*$ is a dual operator system if and only if $\mu_{Y^*}$ is completely isometric. 
	\end{lem}

	%\medskip
	
	\begin{ntn}\label{ntn:predual-map}
		Let $X$ be a SMOS, $Y$ be a complete SMOS, and $V$ and $W$ be dual operator systems. 
		
		\smnoind
		(a) Denote by $X\check{\ }$ the operator subsystem $j_X(X)\subseteq  \M(X)$, and by $X_\os$ the norm-closure of $X\check{\ }$ in $\M(X)$, equipped with the induced operator system structure.

		\smnoind
		(b) We will use $\iota_{X}:X\to X\check{\ }\subseteq X_\os$, instead of $j_X$, to denote  the canonical map from $X$ onto $X\check{\ }$.
		
		\smnoind
		(c) As in Lemma \ref{lem:Q}, we identify the operator subsystem 
		$$(Y^*)\check{\ } = j_{Y^*}(Y^*)\subseteq \M(Y^*)$$  
		with the operator subsystem $\mu_{Y^*}(Y^*)\subseteq \N(Y^*)$. 
Thus, we may also identify $\iota_{Y^*}$ as the map $\mu_{Y^*}$ from $Y^*$ to $(Y^*)\check{\ }$. 
		Under this identification, 
		\begin{quotation}
			we consider $Y^\rd$ to be the weak$^*$-closure of $\iota_{Y^*}(Y^*)$ in $\N(Y^*)$;
		\end{quotation}
		in other words, $Y^\rd$ is the completion of a quotient of $Y^*$ under the vector topology induced from $\mu_{Y^*}$, equipped with the dual operator system structure induced from $\N(Y^*)$. 
		For simplicity, we denote $Y^{\rd\rd} := (Y^\rd)^\rd$.

\smnoind
(d) We always use $V_*$ to denote the fixed predual of $V$. 
		
		\smnoind
		(e) For $\phi\in \Morc_w(V,W)$, we let $\phi_*:W_*\to V_*$ 
		be the bounded linear map given by $\phi_*= \phi^*|_{W_*}$. 
	\end{ntn}

	%\medskip

	\section{operator system preduals of dual operator systems}
	
	%\medskip
	
	Let us begin with the following lemma. 
	The statements in this lemma can be found in either \cite{JN2}, \cite{Ng} or \cite{Ng1} (except that part (a) is formally different from the corresponding result in \cite{Ng}, and part (d) is a generalization of the corresponding result in \cite{JN2}).   
	We give their arguments here for completeness.

	%\medskip
	
	\begin{lem}\label{lem:JN2}
		Let $X$ be a SMOS, $Y$ be a complete SMOS, $V$ be a dual operator system, and $S$ and $T$ be complete operator systems.
		As in parts (b) and (c) of Notation \ref{ntn:predual-map}, we identify $\iota_X = j_X$, $\iota_{Y^*} = \mu_{Y^*}$, $\iota_{S^*} = \mu_{S^*}$ and $\iota_{T^*} = \mu_{T^*}$.
		
		\smnoind
		(a) For any $\varphi \in \Morc(X,T)$, there exists a unique map $\varphi_\os\in \Morc(X_\os, T)$ such that $\varphi_\os\circ \iota_X = \varphi$.
		
		\smnoind
		(b) For any $\psi \in \Morc_w(Y^*,V)$, there exists a unique map $\overline{\psi}\in \Morc_w(Y^\rd, V)$ such that $\overline{\psi}\circ \iota_{Y^*} = \psi$.
		
		\smnoind
		(c) For any $\varphi\in \Morc(S,T)$, there exists a unique map $\varphi^\rd\in \Morc_w(T^\rd,S^\rd)$ such that $\varphi^\rd\circ \iota_{T^*} = \iota_{S^*}\circ \varphi^*$; in other words, the following diagram commutes: 
		\begin{equation}\label{eqt:def-psi-d}
			\xymatrix{
				& T^* \ar[d]_{\iota_{T^*}} \ar[r]^{\varphi^*} & S^* \ar[d]^{\iota_{S^*}} \\
				& T^\rd \ar[r]^{\varphi^\rd} & S^\rd
			}
		\end{equation}  
		
		\smnoind
		(d) $Y$ has a generating cone if and only if  $\iota_{Y^*}:Y^*\to \N(Y^*)$ is bounded below.
		In this case, $\iota_{Y^*}:Y^*\to Y^\rd$ is a weak$^*$-homeomorphic complete order isomorphism. 
		
		\smnoind
		(e) If there exists a completely positive complete isometry $\varphi:X\to T$, then $X$ is an operator system.
		
		\smnoind
		(f) $T^{**}$ is a dual operator system and $T^*$ has a generating cone. 
		
		\smnoind
		(g) If $T$ has a generating cone, then $T^\rd$ has a generating cone.
	\end{lem}
	\begin{proof}
		(a) For each $n\in \BN$, compositions with  $\varphi$ induce a continuous map $\varphi_n:\Morc(T,\BM_n)\to \Morc(X,\BM_n)$. 
		From these maps, we have a unital $^*$-homomorphism 
		$$\widehat{\varphi}:\M(X)\to \M(T)$$ 
		satisfying  $\widehat{\varphi}(j_X(X))\subseteq j_T(T)$.
		Moreover, parts (a) and (b) of Lemma \ref{lem:norm-ind-by-iota} tell us that $j_T:T\to  \M(T)$ is a completely isometric complete order monomorphism. 
		Hence, $\widehat{\varphi}(X_\os) \subseteq j_T(T)$ (as $j_T(T)$ is norm closed), and $\varphi_\os: = j_T^{-1}\circ \widehat{\varphi}|_{X_\os}$ will satisfy the requirement.   
		The uniqueness comes from the norm density of $j_X(X)$ in $X_\os$. 
		
		\smnoind
		(b) Lemmas \ref{lem:image-weak-st-cont-bdd-below} and \ref{lem:mu-compl-ord-mono} as well as Relation \eqref{eqt:iota-inv=j-inv} tell us that $\mu_V$ is a weak$^*$-homeomorphic completely isometric complete order isomorphism from $V$ onto $\mu_V(V)$. 
		As in the argument of part (a) above, $\psi$ induces a unital normal $^*$-homomorphism from $ \N(Y^*)$ to $ \N(V)$ sending $ \mu_{Y^*}(Y^*)$ to $\mu_V(V)$.
		The conclusion then follows from a similar line of argument as in part (a) above (note that the uniqueness comes from the weak$^*$-density of $\mu_{Y^*}(Y^*)$ in $Y^\rd$).

		\smnoind
		(c) This part follows from part (b) above (because Notation \ref{ntn:predual-map}(c) implies that $S^\rd$ is a dual operator system). 
		
		\smnoind
		(d) Suppose that $Y$ has a generating cone. 
		The weak$^*$-continuity of $\mu_{Y^*}:Y^*\to  \N(Y^*)$ produces a map $\Psi\in \Morc\big( \N(Y^*)_*, Y\big)$ satisfying $\mu_{Y^*} = \Psi^*$. 
		Let $v\in B^+_Y$. 
		Define $\omega_v\in  \N(Y^*)_*$ such that $\omega_v$ vanishes on ${\bigoplus}_{n\geq 2}^{\ell^\infty}\, \ell^\infty\big(B^+_{\BM_n(Y)};\BM_n\big)$ and that 
		\begin{equation*}
			\omega_v(\phi) = \phi(v)\qquad (\phi\in \ell^\infty(B^+_Y;\BC)). 
		\end{equation*}
		It is not hard to check that $\Psi(\omega_v) =v$. 
		The generating cone assumption implies that $\Psi$ is surjective and hence is an open map.
		Thus, $\mu_{Y^*} = \Psi^*$ is bounded below. 
		
		Conversely, assume that $\iota_{Y^*}$ is bounded below.
		Let $\Theta:  \N(Y^*)\to \ell^\infty(B^+_Y, \BC)$ be the truncation map, which is a surjective normal unital $^*$-homomorphism. 
		Set 
		$$\Phi:= \Theta \circ \iota_{Y^*}|_{Y^*_\sa}: Y^*_\sa  \to \ell^\infty(B^+_Y;\BR).$$
		The bounded below assumption implies that there exists $\kappa > 0$ with 
		$$\|\iota_{Y^*}(f)\| > \kappa \|f\| \qquad (f\in Y^*_\sa).$$
		Consider  $f\in Y^*_\sa$ with $\|f\| = 1$. 
		Since $\|\iota_{Y^*}(f)\| > \kappa $, one can find $m\in \BN$ and $\psi\in \Morc_w(Y^*,\BM_m)$ satisfying
		$$\|\psi(f)\| > \kappa.$$
		As $\psi(f)\in (\BM_m)_\sa$, there is $\xi\in \BC^m$ such that  $\|\xi\| =1$ and $\|\psi(f)\| = |\langle \psi(f)\xi, \xi\rangle|$. 
		Define $\psi_\xi:Y^*\to \BC$ by 
		$$\psi_\xi(g) := \langle \psi(g)\xi, \xi\rangle \qquad (g\in Y^*).$$
		Then $\psi_\xi\in \Morc_w(Y^*,\BC) = B^+_Y$, and 
		$\kappa < \|\psi(f)\| = |\psi_\xi(f)| \leq \|\Phi(f)\|$. 	
		Consequently, $\Phi$ is bounded below. 
		Lemma \ref{lem:image-weak-st-cont-bdd-below} then implies that $\Phi(Y^*_\sa)$ is a weak$^*$-closed subspace of $\ell^\infty(B^+_Y;\BR)$. %and $\Phi$ is a weak$^*$-homeomorphism from $Y_\sa^*$ to $\Phi(Y_\sa^*)$. 
		Denote 
		$$E:= \big\{\omega\in \ell^1(B^+_Y;\BR): \omega(\Phi(Y^*_\sa)) = \{0\} \big\}.$$ 
		Then $E$ is a closed subspace of $\ell^1(B^+_Y;\BR)$ such that $\Phi(Y^*_\sa) \cong (\ell^1(B^+_Y;\BR)/E)^*$ and the restriction of $\sigma\big(\ell^\infty(B^+_Y;\BR), \ell^1(B^+_Y;\BR)\big)$ on $\Phi(Y^*_\sa)$ is $\sigma\big(\Phi(Y^*_\sa),\ell^1(B^+_Y;\BR)/E\big)$. 
		As $\Phi:  Y^*_\sa\to \Phi(Y^*_\sa)$ is a  weak$^*$-homeomorphism (see Lemma \ref{lem:image-weak-st-cont-bdd-below}), we learn that $\Phi_*$ induces a Banach space isomorphism from $\ell^1(B^+_Y;\BR)/E$ onto $Y_\sa$, which gives $\Phi_*\big(\ell^1(B^+_Y;\BR)\big) = Y_\sa$. 
		Since $\Phi$ is positive, the map $\Phi_*: \ell^1(B^+_Y;\BR)\to Y_\sa$ is a positive surjection. 
		From this, together with $\ell^1(B^+_Y;\BR) = \ell^1(B^+_Y;\BR)_+ - \ell^1(B^+_Y;\BR)_+$, we see that $Y_\sa = Y_+ - Y_+$. 
		These verify the asserted equivalence. 
		
		For the second statement, assume that $Y$ has a generating cone.
		We learn from Lemma \ref{lem:image-weak-st-cont-bdd-below} and the first statement that $\iota_{Y^*}:Y^*\to \iota_{Y^*}(Y^*)$ is a weak$^*$-homeomorphism and thus $\iota_{Y^*}(Y^*)$ is weak$^*$-closed in $\N(Y^*)$. 
		This gives $\iota_{Y^*}(Y^*) = Y^\rd$. 
		Moreover, Relation \eqref{eqt:iota-inv=j-inv} tells us that  $\iota_{Y^*}:Y^*\to \iota_{Y^*}(Y^*)$ is a complete order isomorphism.

		\smnoind
		(e) Since $\varphi_\os\circ \iota_X = \varphi$ (see part (a) above) and $\varphi_\os$ is a complete contraction, we know that $\iota_X$ is a complete isometry. 
		The required conclusion then follows from Lemma \ref{lem:norm-ind-by-iota}(b).     
		
		\smnoind
		(f) By considering the bidual of a completely isometric complete order monomorphism $\Phi:T\to \CL(\KH)$, we learn from part (e) above that $T^{**}$ is an operator system.
		Now, it follows from Lemma \ref{lem:norm-ind-by-iota}(b) that $\iota_{T^{**}}$ is a complete isometry (see Notation \ref{ntn:predual-map}(b)).  
		Hence, Lemma \ref{lem:mu-compl-ord-mono}  (see also Notation \ref{ntn:predual-map}(c)) implies that $T^{**}$ is a dual operator system.
		
		On the other hand, as $\iota_{T^{**}}$ is completely isometric, we learn from part (d) above that $T^*$ has a generating cone.

\smnoind		
(g) Part (d) above implies that $\iota_{T^*}: T^* \to T^\rd$ is a complete order isomorphism. 
Part (f) above then ensures that $T^\rd$ has a generating cone.
	\end{proof}
	
	%\medskip
	
	Part (a) above means that $X_\os$ is the ``universal complete operator system envelope'' of $X$, and part (b) means that $Y^\rd$ is the ``universal dual operator system envelope'' of $Y^*$, while part (c) implies that the construction $S\mapsto S^\rd$ is a contravariant functor, which was called the  ``dual functor'' in \cite{JN2}.
	
%\medskip

\begin{cor}\label{cor:compl-op-sys}
Let $S$ be an operator system, and $\tilde S$ be its operator space completion, equipped with the induced involution. 
If we define $\BM_\infty(\tilde S)_+$ to be the norm-closure of $\BM_\infty(S)_+$ in $\BM_\infty(\tilde S)$, then $\tilde S$ becomes an operator system. 
\end{cor}
\begin{proof}
Let $\tilde j_S: \tilde S\to \M(S)$ be the continuous extension of $j_S:S\to \M(S)$. 
As $j_S$ is a complete isometry (see Lemma \ref{lem:norm-ind-by-iota}(b)), we know that $\tilde j_S$ is a complete isometry.
Moreover, the definition of $\BM_\infty(\tilde S)_+$ ensures that $\tilde j_S$ is completely positive.
Now, the conclusion follows from Lemma \ref{lem:JN2}(e).
\end{proof}

%\medskip

\begin{cor}\label{cor:Y*-os}
	If $Y$ is a complete SMOS with a generating cone, then $j_{Y^*}(Y^*) = (Y^*)_\os$, and $Q_{Y^*}|_{(Y^*)_\os}$ is a completely isometric complete order isomorphism from $(Y^*)_\os$ onto $Y^\rd$. 
\end{cor}
\begin{proof}
Since $Y^\rd$ is norm closed in $\N(Y^*)$, Relation \eqref{eqt:mu=Q-j} produces the following commuting diagram 
\begin{equation*}
	\xymatrix{
		& Y^* \ar[d]_{j_{Y^*}} \ar[r]^{\mu_{Y^*}} & Y^\rd\\
		& (Y^*)_\os \ar[ur]_{Q_{Y^*}|_{(Y^*)_\os}} 
	}
\end{equation*} 
By Lemma \ref{lem:JN2}(d), the map $\mu_{Y^*}:Y^*\to Y^\rd$ is a weak$^*$-homeomorphic complete order isomorphism.
The above diagram then tells us that $j_{Y^*}:Y^*\to (Y^*)_\os$ is bounded below and thus $j_{Y^*}(Y^*)$ is norm-closed. 
This gives the first claim. 
The second claim follows from the first one, together with Lemma \ref{lem:Q}. 
\end{proof}

%\medskip

In the case when the complete SMOS $Y$ does not have a generating cone, we know  from Lemma \ref{lem:Q} that the completely positive map $Q_{Y^*}|_{(Y^*)_\os}: (Y^*)_\os\to Y^\rd$ is completely isometric. 
However, we do not know whether it is a complete order monomorphism. 
This issue is related to the following question: 
\begin{quote}
{\it when $T$ is a dense  operator subsystem of a complete operator system $S$, will the matrix cone of $T$ be dense in the matrix cone of $S$?}
\end{quote}
Note that this question has a positive answer when $T$ is unital or is a dense $^*$-subalgebra of a $C^*$-algebra. 
The following tells us that this question has a positive answer when $T = X\check{\ }$ and $S= X_\os$ for a SMOS $X$. 
	
	%\medskip
	
	\begin{lem}\label{lem:compl}
		If $X$ is a SMOS, then $j_X^{(\infty)}\big(\BM_\infty(X)_+\big) = \BM_\infty\big(j_X(X)\big)\cap \BM_\infty\big(\M(X)\big)_+$ and is norm-dense in $\BM_\infty(X_\os)_+$. 
	\end{lem}
	\begin{proof}
The equality $j_X^{(\infty)}\big(\BM_\infty(X)_+\big) = \BM_\infty\big(j_X(X)\big)\cap \BM_\infty\big(\M(X)\big)_+$ is a direct consequence of Lemma \ref{lem:norm-ind-by-iota}(a). 
Let us denote by $Z$ the completion of the operator space $j_X(X)$, equipped with the involution induced from the one on $j_X(X)$ and with its matrix cone $\BM_\infty(Z)_+$ being the norm-closure of $\BM_\infty(j_X(X))_+$ in $\BM_\infty(Z)$. 
By Corollary \ref{cor:compl-op-sys}, we know that $Z$ is a complete operator system. 

It suffices to show the identity map $\iota:Z \to X_\os$ is a complete order isomorphism. 
Indeed, it is clear that $\iota$ is completely positive, because the matrix cone $\BM_\infty(X_\os)_+$ is norm-closed and contains $\BM_\infty\big( (j_X(X))_+\big)$. 
Conversely, since $j_X\in \Morc(X, j_X(X)) \subseteq \Morc(X, Z)$, by Lemma \ref{lem:JN2}(a), there exists a unique map $(j_X)_\os\in \Morc(X_\os, Z)$ such that the following diagram
\begin{equation*}
	\xymatrix{
		& X \ar[d]_{j_X} \ar[r]^{j_X} & Z\\
		& X_\os \ar[ur]_{(j_X)_\os} 
	}
\end{equation*} 
commutes. This means that the completely positive map $(j_X)_\os$ is the identity map, and hence is the inverse of $\iota$. 
Therefore, $\iota$ is a complete order isomorphism. 
\end{proof}

%\medskip
	
%	The proof of Lemma \ref{lem:compl} above and Lemma \ref{lem:norm-ind-by-iota}(b) give a more direct way to describe $T_\os$ (without going through $j_T$)  when $T$ is an operator system:   
%	$T_\os$ is the completion of $T$, equipped with the induced involution such that  $\BM_\infty(T_\os)_+$ is the norm-closure of $\BM_\infty(T)_+$ in $\BM_\infty(T_\os)$. 

	%\medskip
	
	For a dual operator system $V$, we set
	\begin{equation*}\label{eqt:def-pred-oper-sys}
		V_\# := (V_*)_\os.
	\end{equation*}
	By Lemma \ref{lem:compl}, the matrix cone $\BM_\infty(\iota_{V_*}(V_*))_+ = j_{V_*}^{(\infty)}\big(\BM_\infty(V_*)_+\big)$ is norm dense in $\BM_\infty(V_\#)_+$. 
	
	%\medskip
	
	Observe that the canonical embedding $\kappa_{V_*}: V_* \to V^*$ gives another canonical way to associate a complete operator system with the SMOS $V_*$; namely,  the norm-closure of $\iota_{V^*}(\kappa_{V_*}(V_*))$ in $V^\rd\subseteq \N(V^*)$. 
	In the following, we will show that these two methods actually produce the same operator system. 
	
	%\medskip
	
	In fact, by Lemma \ref{lem:JN2}(a), there exists a unique completely positive complete contraction $\Gamma_{V_\#}: V_\# \to  V^\rd$ such that the diagram 
	\begin{equation}\label{eqt:iota-ext-Gamma-circ-iota}
		\xymatrix{
			& V_* \ar[d]_{\iota_{V_*}} \ar[r]^{\kappa_{V_*}} & V^* \ar[d]^{\iota_{V^*}}\\
			& V_\# \ar[r]^{\Gamma_{V_\#}} & V^\rd 
		}
	\end{equation}
	commutes. 
	Since $\Morc(V_*, \BM_n) = B^+_{\BM_n(V)} = \Morc_w(V^*, \BM_n)$ (see 
	\eqref{eqt:CCP-B+}), one has 
	$$ \M(V_*) \subseteq  \N(V^*).$$
	Thus, $\iota_{V_*}=j_{V_*}$ is the evaluation map from $V_*$ to $ \N(V^*)$.
	On the other hand, as in Notation \ref{ntn:predual-map}(c),  $\iota_{V^*}$ can be identified with the evaluation map $\mu_{V^*}:V^* \to  \N(V^*)$. 
	From these and the uniqueness of the map $\Gamma_{V_\#}$ that satisfies \eqref{eqt:iota-ext-Gamma-circ-iota}, we know that $\Gamma_{V_\#}$ is the restriction of the inclusion map $i:  \M(V_*)\to  \N(V^*)$ to the norm-closure of $\iota_{V_*}(V_*)$ in $ \M(V_*)$. 
	
	%\medskip
	
	Let us summarize the above as follows.

	%\medskip
	
	\begin{lem}\label{lem:predual-dual-os}
		Let $V$ be a dual operator system. 
		The map $\Gamma_{V_\#}\in \Morc(V_\#, V^\rd)$ as above is a completely isometric compete order monomorphism. 
	\end{lem}
	
	%\medskip

In the following, for a dual operator system $V$,  
$$\text{we will regard $\iota_{V_*}$ as a map in }\Morc (V_*, V_\#).$$
By Lemma \ref{lem:JN2}(b), one obtains a map $\overline{\iota_{V_*}^*}\in \Morc_w\big((V_\#)^\rd, V\big)$ such that 
\begin{equation}\label{eqt:iota-bar-iota}
	\xymatrix{
		& (V_\#)^* \ar[d]_{\iota_{(V_\#)^*}} \ar[r]^{\iota_{V_*}^*} & V\\
		& (V_\#)^\rd \ar[ur]_{\overline{\iota_{V_*}^*}} 
	}
\end{equation}
is a commuting diagram. 

	%\medskip
	
%	Our main results concern with the case when $V$ has a generating cone.
%	In this case, we have the following. 

	%\medskip
	
	\begin{lem}\label{lem:predual1}
		Let $V$ be a dual operator system. 
		
		\smnoind
		(a) The dual map $\iota_{V_*}^*\in \Morc((V_\#)^*, V)$ is a complete order monomorphism. 
		
		\smnoind
		(b) If $V_\#$ has a generating cone, then $\overline{\iota_{V_*}^*}:(V_\#)^\rd \to V$ is a complete order monomorphism. 
		
		\smnoind
		(c) Suppose that $V$ has a generating cone. 
		Then $\iota_{V_*}: V_* \to V_\#$ is a complete order isomorphism and  $V_\#$ has a generating cone. In particular, $V_\# = \iota_{V_*}(V_*)$.
	\end{lem}
	\begin{proof}
		(a) Since $\iota_{V_*}(V_*)$ is norm dense in $V_\#$, the map $\iota_{V_*}^*$ is injective. 
		It remains to show that 
		\begin{equation*}\label{eqt:iota-V-*}
			(\iota_{V_*}^*)^{(\infty)}\big(\BM_\infty\big((V_\#)^*\big)_+\big) = (\iota_{V_*}^*)^{(\infty)}\big(\BM_\infty\big((V_\#)^*\big)_\sa\big)\cap \BM_\infty(V)_+.
		\end{equation*}
		In fact, the complete positivity of $\iota_{V_*}^*$ implies that the left hand side is a subset of the right hand side. 
		For the opposite inclusion, consider $u\in \BM_\infty((V_\#)^*)_\sa$ satisfying $(\iota_{V_*}^*)^{(\infty)}(u)\in \BM_\infty(V)_+$. 
		We have  
		$$\big( u, (\iota_{V_*})^{(\infty)}(\omega) \big) = \big((\iota_{V_*}^*)^{(\infty)}(u), \omega \big) \geq 0 \qquad (\omega\in \BM_\infty(V_*)_+).$$
		Since $(\iota_{V_*})^{(\infty)}(\BM_\infty(V_*)_+)$   is norm dense in $\BM_\infty(V_\#)_+$ (by Lemma \ref{lem:compl}), the displayed relation above implies that $u\in \BM_\infty\big((V_\#)^*)_+$. 
		
		\smnoind
		(b) Lemma \ref{lem:JN2}(d) implies that $\iota_{(V_\#)^*}: (V_\#)^* \to (V_\#)^\rd$ is a complete order isomorphism. 
		We then know from this, part (a) above and  \eqref{eqt:iota-bar-iota} that $\overline{\iota_{V_*}^*}$ is a complete order monomorphism.
		
		\smnoind
		(c) By Lemma \ref{lem:JN2}(d), we know that $\iota_{V^*}$ is bounded below. 
		Since $\kappa_{V_*}$ is completely isometric, \eqref{eqt:iota-ext-Gamma-circ-iota} implies that $\iota_{V_*}$ is bounded below. 
		Therefore, $\iota_{V_*}(V_*)$ is complete.
		This tells us that 
		$$V_\# = (V_*)_\os = \iota_{V_*}(V_*)$$ 
		and $\iota_{V_*}$ is a Banach space isomorphism from $V_*$ onto $V_\#$. 
		Thus, we conclude from Lemma \ref{lem:norm-ind-by-iota}(a) that $\iota_{V_*}: V_* \to V_\#$ is a complete order isomorphism. 
		Moreover, 
		$$\iota_{V_*}^*: (V_\#)^* \to V$$ 
		will be a Banach space isomorphism, and 
		we learn from \eqref{eqt:iota-bar-iota} that $\iota_{(V_\#)^*}$ is bounded below (because $\overline{\iota_{V_*}^*}$ is a complete contraction). 
		Therefore, Lemma \ref{lem:JN2}(d) tells us that $V_\#$ has a generating cone. 
	\end{proof}

	%\medskip
	
	Suppose that $T$ is a complete operator system.
	By Lemma \ref{lem:JN2}(f), we know that  $T^{**}$ is a dual  operator system.
	Hence, Lemma \ref{lem:JN2}(b) produces a map $\overline{\iota_{T^*}^*}\in \Morc_w(T^{\rd\rd},T^{**})$ with the diagram  
	\begin{equation}\label{eqt:iota-T*-*}
		\xymatrix{
			& (T^\rd)^* \ar[d]_{\iota_{(T^\rd)^*}} \ar[r]^{	\iota_{T^*}^*} & T^{**}\\
			& T^{\rd\rd} \ar[ur]_{\overline{\iota_{T^*}^*}} 
		}
	\end{equation}
	being commute. 
	Moreover, Lemma \ref{lem:JN2}(a) gives $(\iota_{T^*})_\#\in \Morc((T^{\rd})_\#,T)$ with 
	\begin{equation}\label{eqt:iota-T^d_*}
		\xymatrix{
			& (T^\rd)_* \ar[d]_{\iota_{(T^\rd)_*}} \ar[r]^{	(\iota_{T^*})_*} & T\\
			& (T^\rd)_\# \ar[ur]_{(\iota_{T^*})_\#} 
		}
	\end{equation}
	being a commuting diagram. 
	
	%\medskip
	
	%\medskip
	
	\begin{lem}\label{lem:equiv-tight}
		Let $T$ be a complete operator system with a generating cone. 
		
		\smnoind
		(a) $\iota_{(T^\rd)^*}$, $\iota_{T^*}^*$ and $\overline{\iota_{T^*}^*}$ are  weak$^*$-homeomorphic complete order isomorphisms.

		\smnoind
		(b) $(\iota_{T^*})_*$, $\iota_{(T^\rd)_*}$ and $(\iota_{T^*})_\#$ are complete order isomorphisms.	
		
		\smnoind
		(c) $T^{\rd\rd}$ and $T^{**}$ have generating cones. 
	\end{lem}
	\begin{proof}
		By parts (d)  and (g) of Lemma \ref{lem:JN2}, the map $\iota_{T^*}:T^*\to T^\rd$ is a weak$^*$-homeomorphic complete order isomorphism and $T^\rd$ has a generating cone. 
		
		\smnoind
		(a) From the above and Lemma \ref{lem:JN2}(d), both $\iota_{T^*}^*$ and $\iota_{(T^\rd)^*}$ are  weak$^*$-homeomorphic complete order isomorphisms. 
		We then conclude from \eqref{eqt:iota-T*-*} that $\overline{\iota_{T^*}^*}$ is a weak$^*$-homeomorphic complete order isomorphism.

		\smnoind
		(b) Since $\iota_{T^*}:T^*\to T^\rd$ is a weak$^*$-homeomorphic complete order isomorphism, the map  $(\iota_{T^*})_*: (T^\rd)_*\to T$ is a complete order isomorphism. 
		Moreover, as the dual operator system $T^\rd$ has a generating cone, we learn from Lemma \ref{lem:predual1}(c) that $\iota_{(T^\rd)_*}: (T^\rd)_*\to (T^\rd)_\#$ is a complete order isomorphism, and hence so is 
		$(\iota_{T^*})_\#$ (because of  \eqref{eqt:iota-T^d_*}). 
		
		\smnoind
		(c) Lemma \ref{lem:JN2}(g) implies that $T^{\rd\rd}$ has a generating cone. 
		Hence,  $T^{**}$ also has a generating cone (as $\overline{\iota_{T^*}^*}$ is a complete order isomorphism by part (a) above). 
	\end{proof}
	
	%\medskip
	
	\begin{defn}\label{defn:d-dual}
		Let $T$ be a complete operator system with a generating cone. 
		Then $T$ is said to be \emph{tight} if the map $\overline{\iota_{T^*}^*}:T^{\rd\rd} \to T^{**}$ is completely isometric.
	\end{defn}

By Lemma \ref{lem:equiv-tight}(a), $T$ is tight if and only if one can identify $T^{\rd\rd}$ and $T^{**}$ as the same the dual operator system, via  $\overline{\iota_{T^*}^*}$. 
	%\medskip
	
	%\medskip
	
	The following is our first main result. 
	
	%\medskip
	
	\begin{thm}\label{thm:os-predual-tight}
		(a) If 	$V$ is a dual operator system with a generating cone, then $V_\#$ is tight. 
		
		\smnoind
		(b) Let $T$ be a complete operator system with a generating cone. 
		The following statements are equivalent. 
		\begin{enumerate}[(S1)]
			\item $T$ is tight. 
			
			\item $(\iota_{T^*})_\#:  (T^{\rd})_\#\to T$ is completely isometric.
			
			\item $(\iota_{T^*})_\#:  (T^{\rd})_\#\to T$ is a completely isometric complete order isomorphism.
		\end{enumerate}
	\end{thm}
	\begin{proof}
		(a) %Consider the maps 	$\iota_{V_*}\in \Morc(V_*, V_\#)$ and $\iota_{(V_\#)^*}\in \Morc((V_\#)^*, (V_\#)^\rd)$.
		First of all, one has 
		\begin{align*}
			\overline{\iota_{(V_\#)^*}^*} \circ (\overline{\iota_{V_*}^*})^\rd\circ \Gamma_{V_\#}\circ \iota_{V_*} 
			& =\ \overline{\iota_{(V_\#)^*}^*} \circ (\overline{\iota_{V_*}^*})^\rd\circ \iota_{V^*}\circ \kappa_{V_*}\\
			& =\ \overline{\iota_{(V_\#)^*}^*} \circ \iota_{((V_\#)^\rd)^*}\circ (\overline{\iota_{V_*}^*})^*\circ \kappa_{V_*}\\
			& =\ \iota_{(V_\#)^*}^*\circ (\overline{\iota_{V_*}^*})^*\circ \kappa_{V_*}\\
			&  =\  (\overline{\iota_{V_*}^*}\circ \iota_{(V_\#)^*})^* \circ \kappa_{V_*}\\
			&  =\  \iota_{V_*}^{**}\circ \kappa_{V_*}
			\ = \ \kappa_{V_\#}\circ \iota_{V_*};
		\end{align*}
		here, the first equality follows from \eqref{eqt:iota-ext-Gamma-circ-iota}, the second one follows from \eqref{eqt:def-psi-d} (for $\varphi = \overline{\iota_{V_*}^*}$), the third one from \eqref{eqt:iota-T*-*} (for $T= V_\#$), the fifth one from \eqref{eqt:iota-bar-iota}, and the last one from \eqref{eqt:kappa} (for $\phi = \iota_{V_*}$).
		As both $\overline{\iota_{(V_\#)^*}^*} \circ (\overline{\iota_{V_*}^*})^\rd\circ \Gamma_{V_\#}$ and $\kappa_{V_\#}$ are bounded, and $\iota_{V_*}(V_*)$ is norm-dense in $V_\#$, we see that the following diagram commutes:
		\begin{equation*}
			\xymatrix{
				& V_\# \ar[d]_{\Gamma_{V_\#}} \ar[r]^{\kappa_{V_\#}} & (V_\#)^{**}\\
				& V^\rd \ar[r]^{\big(\overline{\iota_{V_*}^*}\big)^\rd} & (V_\#)^{\rd\rd} \ar[u]_{\overline{\iota_{(V_\#)^*}^*}}
			}
		\end{equation*}

		In order to show that $V_\#$ is tight, i.e., the map $\Theta:=\overline{\iota_{(V_\#)^*}^*}$ is completely isometric, we set  $\Lambda:=\big(\overline{\iota_{V_*}^*}\big)^\rd\circ \Gamma_{V_\#}$  and $S:= V_\#$. 
		The diagram above means that 
		\begin{equation}\label{eqt:tight}
			\Theta \circ \Lambda = \kappa_S.
		\end{equation}
		
		By Lemmas \ref{lem:predual1}(c) and \ref{lem:equiv-tight}(a) (for $T=V_\#$), one knows that the complete contraction $\Theta$ is a weak$^*$-continuous complete order isomorphism. 
		Consider $n \in \BN$. 
		Let $x\in \BM_n(S^{\rd\rd})$ with $\big\|\Theta^{(n)}(x)\big\| \leq 1$. 
		We obtain from Lemma \ref{lem:unit-ball-weak-st-dense}(a) a net $\{y_i\}_{i\in \KI}$ in $B_{\BM_n(S)}$ such that 
		\begin{quotation}
			$\{\kappa_{S}^{(n)}(y_i)\}_{i\in \KI}$ will  $\sigma(\BM_n(S^{**}), \BM_n(S^*))$-converge to $\Theta^{(n)}(x)\in B_{\BM_n(S^{**})}$.
		\end{quotation}
		Since $\Lambda^{(n)}$ is contractive, Lemma \ref{lem:unit-ball-weak-st-dense}(c) produces a subnet of $\big\{\Lambda^{(n)}(y_i)\big\}_{i\in \KI}$ that  $\sigma\big(\BM_n(S^{\rd\rd}), \BM_n((S^{\rd\rd})_*)\big)$-converges to some element $z\in B_{\BM_n(S^{\rd\rd})}$. 
		The weak$^*$-continuity of $\Theta$ and \eqref{eqt:tight} ensure that $\Theta^{(n)}(z)  = \Theta^{(n)}(x)$. 
		As $\Theta$ is injective, we know that $x = z\in B_{\BM_n(S^{\rd\rd})}$. This implies that  $\Theta^{(n)}$ is isometric. 
		
		\smnoind
		(b) $(S1)\Rightarrow (S2)$. 
		By \eqref{eqt:kappa} (for $\phi=(\iota_{T^*})_*$), the following diagram
		\begin{equation}\label{eqt:kappa-iota}
			\xymatrix{
				& (T^\rd)_* \ar[d]_{\kappa_{(T^\rd)_*}} \ar[r]^{(\iota_{T^*})_*} & T \ar[d]^{\kappa_T}\\
				& (T^{\rd})^* \ar[r]^{\iota_{T^*}^*} & T^{**} 
			}
		\end{equation}
		commutes.
Moreover, it follows from \eqref{eqt:iota-ext-Gamma-circ-iota} (for $V= T^\rd$) that 
\begin{equation}\label{eqt:Gamma-iota}
	\xymatrix{
		& (T^\rd)_* \ar[d]_{\kappa_{(T^\rd)_*}} \ar[r]^{\iota_{(T^\rd)_*}} & (T^\rd)_\# \ar[d]^{\Gamma_{(T^\rd)_\#}}\\
		& (T^\rd)^* \ar[r]^{\iota_{(T^\rd)^*}} & T^{\rd\rd}
	}
\end{equation}
is a commuting diagram. 
By parts (a) and (b) of Lemma \ref{lem:equiv-tight}, we know that both $\iota_{(T^\rd)_*}$ and $\iota_{(T^\rd)^*}$ are bijective. 
Thus, \eqref{eqt:Gamma-iota} becomes 
\begin{equation*}
	\xymatrix{
		& (T^\rd)_\# \ar[d]_{\Gamma_{(T^\rd)_\#}} \ar[r]^{\iota_{(T^\rd)_*}^{-1}} & (T^\rd)_* \ar[d]^{\kappa_{(T^\rd)_*}} \\
		& T^{\rd\rd} \ar[r]^{\iota_{(T^\rd)^*}^{-1}} & (T^\rd)^*
	}
\end{equation*}
and this, together with \eqref{eqt:kappa-iota}, produces 
\begin{align*}
\kappa_T\circ (\iota_{T^*})_* \circ \iota_{(T^\rd)_*}^{-1} 
& =  \iota_{T^*}^*\circ \iota_{(T^\rd)^*}^{-1}\circ \Gamma_{(T^\rd)_\#}. 
\end{align*}
Therefore, \eqref{eqt:iota-T*-*} and \eqref{eqt:iota-T^d_*} imply that 
$$\kappa_T\circ (\iota_{T^*})_\# =  \overline{\iota_{T^*}^*}\circ \Gamma_{(T^\rd)_\#}.$$ 
Now, Lemma \ref{lem:predual-dual-os} and the assumption of $\overline{\iota_{T^*}^*}$ being a complete isometry ensure that $(\iota_{T^*})_\#$ is completely isometric.

		\smnoind
		$(S2)\Rightarrow (S3)$. 
		This follows from Lemma \ref{lem:equiv-tight}(b). 
		
		\smnoind
		$(S3)\Rightarrow (S1)$. 
		This follows from part (a) above and the fact that the dual operator system $T^\rd$ has a generating cone (because of Lemma \ref{lem:JN2}(g)). 
	\end{proof}

	%\medskip
	
	In symmetry, we have the following result for dual operator systems.

	%\medskip
	
	\begin{thm}\label{thm:predual-dual}
		Suppose that $V$ is a dual operator system with a generating cone. 
		
		\smnoind
		(a) $(V_\#)^\rd$ is tight. 
		
		\smnoind
		(b) The following statements are equivalent. 
		\begin{enumerate}[(D1)]
			\item $V\cong T^\rd$ as dual operator systems (completely isometrically), for a complete operator system $T$. 
			
			\item $\overline{\iota_{V_*}^*}: (V_\#)^\rd \to V$ is completely isometric.
			
			\item $\overline{\iota_{V_*}^*}: (V_\#)^\rd \to V$ is a weak$^*$-homeomorphic completely isometric complete order isomorphism. 
			
			\item $V$ is tight. 
		\end{enumerate}
	\end{thm}
	\begin{proof}
	Note that the operator system $V_\#$ has a generating cone, because of Lemma \ref{lem:predual1}(c).
Hence, Lemma \ref{lem:predual1}(b) tell us  that  
$$\overline{\iota_{V_*}^*}:(V_\#)^\rd \to V$$ is a complete order monomorphism

\smnoind
(a)	It follows from Lemma \ref{lem:JN2}(f) and Lemma \ref{lem:equiv-tight}(c) that $(V_\#)^{**}$ is a dual operator system having a generating cone.
Moreover, since $((V_{\#})^{**})_{*} = (V_{\#})^{*}$, we have 
$$((V_{\#})^{**})_\# = ((V_{\#})^{*})_\os.$$
On the other hand, Corollary \ref{cor:Y*-os} tells us that $Q_{(V_{\#})^*}|_{((V_{\#})^*)_\os}: ((V_{\#})^{*})_\os \to (V_\#)^\rd$ is a completely isometric complete order isomorphism. 
Hence, $(V_\#)^\rd\cong ((V_{\#})^{**})_\#$ as operator systems, and Theorem \ref{thm:os-predual-tight}(a) imply that it is tight.  

		\smnoind
		(b) $(D1)\Rightarrow (D2)$. 
		Statement (D1) produces a weak$^*$-homeomorphic completely isometric complete order isomorphism $\Lambda:T^\rd \to V$. 
		Let us set 
		\begin{equation}\label{eqt:def-varphi}
		\varphi:= (\Lambda\circ \iota_{T^*})_*\in \Morc(V_*,T).
		\end{equation}
		Lemma \ref{lem:JN2}(a) produces $\varphi_\os\in \Morc\big(V_\#,T\big)$ satisfying 
		\begin{equation}\label{eqt:varphi-check}
			\varphi_\os \circ \iota_{V_*} = \varphi.
		\end{equation}
		Moreover, Lemma \ref{lem:JN2}(c) produces 
		$\theta:= (\varphi_\os)^\rd \in \Morc_w\big(T^\rd, (V_\#)^\rd \big)$ 
		such that 
		\begin{equation}\label{eqt:thetat-circ}
		\theta\circ \iota_{T^*} = \iota_{(V_\#)^*}\circ (\varphi_\os)^*.
		\end{equation}
		
		We first claim that  
		\begin{equation}\label{eqt:comp=id}
			\overline{\iota_{V_*}^*}\circ \theta = \Lambda. 
		\end{equation}
		In fact, for $f\in T^*$ and $\omega\in V_*$, one has 
		\begin{align*}
			(\overline{\iota_{V_*}^*}\circ \theta)\big(\iota_{T^*}(f)\big)(\omega)
			& =\  \big(\overline{\iota_{V_*}^*}\big(\iota_{(V_\#)^*}\big((\varphi_\os)^*(f)\big)\big)\big)(\omega)
			\ = \ \big(\iota_{V_*}^*\big(f\circ \varphi_\os\big)\big)(\omega)\\
			& =\ f(\varphi_\os(\iota_{V_*}(\omega)))
			\ =\ f(\varphi(\omega))\\
			& =\ f\big((\Lambda\circ \iota_{T^*})_*(\omega)\big) 
			\ = \ \Lambda\big(\iota_{T^*}(f)\big)(\omega);
		\end{align*}
here, the first equality follows from \eqref{eqt:thetat-circ}, the second one from \eqref{eqt:iota-bar-iota}, the fourth one from \eqref{eqt:varphi-check} and the fifth one from \eqref{eqt:def-varphi}. 
	This means that 
	$$\overline{\iota_{V_*}^*}\circ \theta\circ \iota_{T^*} = \Lambda\circ \iota_{T^*}.$$ 
		Since $\iota_{T^*}(T^*)$ is weak$^*$-dense in $T^\rd$ (see Notation \ref{ntn:predual-map}(c)) and both $\overline{\iota_{V_*}^*}\circ \theta$ as well as $\Lambda$ are weak$^*$-continuous, we obtain \eqref{eqt:comp=id}.  
		
		Now, by the injectivity of $\overline{\iota_{V_*}^*}$ (see the beginning of this proof), the bijectivity of $\Lambda$ as well as \eqref{eqt:comp=id}, one sees that $\theta$ is surjective. 
		Finally, as $\theta$ is a surjective complete contraction and $\Lambda$ is a complete isometry, we know from \eqref{eqt:comp=id} that $\overline{\iota_{V_*}^*}$ is completely isometric. 
		
		\smnoind
		$(D2)\Rightarrow (D3)$. 
		As in the beginning of this proof, $\overline{\iota_{V_*}^*}:(V_\#)^\rd \to V$ is a complete order monomorphism. 
		Moreover, Lemma \ref{lem:predual1}(c) tells us  that  $\iota_{V_*}: V_* \to V_\#$ is a Banach space isomorphism, and so is $\iota_{V_*}^*: (V_\#)^*\to V$. 
		The surjectivity of  $\iota_{V_*}^*$ and  \eqref{eqt:iota-bar-iota} then ensure that $\overline{\iota_{V_*}^*}$ is also surjective, and hence it is a complete order isomorphism. 
		Furthermore, it follows from Lemma \ref{lem:image-weak-st-cont-bdd-below} and the surjectivity of $\overline{\iota_{V_*}^*}$ that it is a weak$^*$-homeomorphism. 
		
		\smnoind
		$(D3)\Rightarrow (D1)$. 
		This is plain. 
		
		\smnoind
		$(D3)\Rightarrow (D4)$. 
This follows from part (a) above.  
		
		\smnoind
		$(D4)\Rightarrow (D2)$. 
		Let us first verify that the following diagram 
		\begin{equation}\label{shorteq}
			\xymatrix{
				& V^{\rd\rd} \ar[r]^{\overline{\iota^{*}_{V^{*}}}} \ar[d]_{\Gamma_{V_{\#}}^\rd} & V^{**} \ar[d]^{\kappa^{*}_{V_{*}}} \\
				& (V_\#)^\rd \ar[r]^{\overline{\iota^{*}_{V_{*}}}} & V
			}
		\end{equation}
		commutes. In fact, let $g\in (V^{\rd})^{*}$.
		We learn from \eqref{eqt:iota-T*-*} (for $T =V$) that 
		\begin{equation*}
			\kappa^{*}_{V_{*}}\circ \overline{\iota^{*}_{V^{*}}}\big(\iota_{(V^{\rd})^{*}}(g)\big) 
			= \kappa^{*}_{V_{*}}\circ \iota^*_{V^{*}}(g)
			 = g\circ \iota_{V^{*}}\circ \kappa_{V_{*}}.
		\end{equation*}
		On the other hand, it follows from \eqref{eqt:def-psi-d} (for $\varphi = \Gamma_{V_{\#}}$) as well as  \eqref{eqt:iota-bar-iota} that 
		\begin{align*}
			\overline{\iota^{*}_{V_{*}}}\circ \Gamma_{V_{\#}}^{\rd}\big(\iota_{(V^{\rd})^{*}}(g)\big)
			& =  \overline{\iota^{*}_{V_{*}}}\big(\iota_{(V_\#)^{*}}\big(\Gamma_{V_{\#}}^*(g)\big)\big)
			 = \iota_{V_{*}}^*\circ \Gamma_{V_{\#}}^*(g) 
			 = g\circ \Gamma_{V_{\#}}\circ \iota_{V_{*}}.
		\end{align*}
As $\iota_{V^{*}}\circ \kappa_{V_{*}} = \Gamma_{V_{\#}}\circ \iota_{V_{*}}$ 
(see \eqref{eqt:iota-ext-Gamma-circ-iota}), the two displayed relations above give 
\begin{equation}\label{eqt:D4}
\kappa^{*}_{V_{*}}\circ \overline{\iota^{*}_{V^{*}}}\circ \iota_{(V^{\rd})^{*}} = \overline{\iota^{*}_{V_{*}}}\circ \Gamma_{V_{\#}}^{\rd}\circ \iota_{(V^{\rd})^{*}}.
\end{equation}
Furthermore, we know from Lemma \ref{lem:JN2}(g) that $V^\rd$ has a generating cone, which ensures  $\iota_{(V^{\rd})^{*}}$ to be surjective (because of Lemma \ref{lem:JN2}(d)). 
Now, \eqref{eqt:D4} tells us that the diagram \eqref{shorteq} commutes.
		
		Let $n\in \BN$ and $[y_{ij}] \in \BM_n(V)$. 
		Since $\kappa_{V_*}^*:V^{**}\to (V_*)^* = V$ is a complete quotient map (see \cite[Corollary 4.1.9]{OPS}), we know that $(\kappa_{V_*}^*)^{(n)}:\BM_n(V^{**})\to \BM_n(V)$ is an exact quotient map (because both $\BM_n(V^{**})$ and $\BM_n(V)$ are dual Banach spaces). 
		Hence, 
		there exists $[v_{ij}]\in \BM_{n}(V^{**})$ such that
		\begin{equation*}
			\big\|[v_{ij}]\big\| = \big\|[y_{ij}]\big\| \quad \text{and} \quad (\kappa_{V_*}^*)^{(n)}\big([v_{ij}]\big)= [y_{ij}].
		\end{equation*}
		On the other hand, as $V$ is tight, there exists $[u_{ij}]\in \BM_{n}(V^{\rd\rd})$ satisfying 
		\begin{equation*}
			\big\|[u_{ij}]\big\| = \big\|[v_{ij}]\big\| \quad \text{and} \quad (\overline{\iota^{*}_{V^{*}}})^{(n)}\big([u_{ij}]\big)= [v_{ij}].
		\end{equation*}
		Let us set $[x_{ij}] := \big(\Gamma_{V_{\#}}^\rd\big)^{(n)}\big([u_{ij}]\big)$. 
		By \eqref{shorteq} and the above, we have 
		\begin{equation*}
			\big\|[x_{ij}]\big\| \leq \big\|[u_{ij}]\big\| = \big\|[y_{ij}]\big\| \quad \text{and} \quad \big(\overline{\iota^{*}_{V_{*}}}\big)^{(n)}\big([x_{ij}]\big)= [y_{ij}].
		\end{equation*}
		However, as $\overline{\iota^{*}_{V_{*}}}$ is a complete contraction, we conclude that 
		$$\big\|[x_{ij}]\big\| = \big\|[y_{ij}]\big\|.$$ 
				This means that $\big(\overline{\iota^{*}_{V_{*}}}\big)^{(n)}$ is an exact quotient map. 
		As $\overline{\iota^{*}_{V_{*}}}$ is also injective (see the beginning of this proof), we know that $\big(\overline{\iota^{*}_{V_{*}}}\big)^{(n)}$ is isometric, for arbitrary $n\in \BN$. 
	\end{proof}
	
	%\medskip
	
	It follows from \cite[Theorem 3.17(b)]{Ng1} that unital complete operator systems are tight (see also \cite[Propostion 6.8(c)]{JN2}).
	This, together with Theorem \ref{thm:os-predual-tight}(b) and Theorem \ref{thm:predual-dual}(b), gives the following corollary.

	%\medskip
	
	\begin{cor}\label{cor:unital}
		(a) If $T$ is a unital complete operator system, then $T\cong (T^{\rd})_\#$ as operator systems (completely isometrically).
		
		\smnoind
		(b) If $V$ is a unital dual operator system, then $V \cong (V_\#)^\rd$ as dual operator systems (completely isometrically). 
	\end{cor}
	
%\medskip

Finally, let $\CE$ be the non-complete operator system  associated with a tolerance relation (as defined in \cite{CvD,CvD2}). 
By Corollary 6.12(a) and Proposition 6.8(c) of \cite{JN2}, 
$$\overline{\iota_{\tilde \CE^*}^*}: \tilde \CE^{\rd\rd} \to \tilde \CE^{**}$$ 
is completely isometric, where $\tilde \CE$ is the complete operator system as in Corollary \ref{cor:compl-op-sys}. 
This means that the complete operator system $\tilde \CE$ is tight, and Theorem \ref{thm:os-predual-tight}(b) gives the following. 

%\medskip 

\begin{cor}\label{cor:toler}
Let $\CE$ be the operator system associated with a tolerance relation. 
Then $\tilde \CE = (\tilde \CE^\rd)_\#$. 
\end{cor}

	%\medskip

\end{document}